\documentclass[11pt]{amsart}

\newtheorem{theorem}{Theorem}[section]

\newtheorem{corollary}[theorem]{Corollary}

\newtheorem{definition}[theorem]{Definition}
\newtheorem{lemma}[theorem]{Lemma}

\newtheorem{proposition}[theorem]{Proposition}
\newtheorem{remark}[theorem]{Remark}

\usepackage[colorlinks, bookmarks=true]{hyperref}
\usepackage{color,graphicx,shortvrb}
\usepackage[active]{srcltx} %SRC Specials for DVI search

\def\11{\textbf{$1$}}

\usepackage{enumerate}
\usepackage{amsmath}
\usepackage{amssymb}
\usepackage[latin 1]{inputenc}

\begin{document}

\title{Weak-local derivations and homomorphisms on C$^*$-algebras}

\author[Ben Ali Essaleh]{Ahlem Ben Ali Essaleh}
\email{ahlem.benalisaleh@gmail.com}
\address{Faculte des Sciences de Monastir, Département de Mathématiques, Avenue de L'environnement, 5019 Monastir, Tunisia}
\curraddr{Departamento de An{\'a}lisis Matem{\'a}tico, Facultad de
Ciencias, Universidad de Granada, 18071 Granada, Spain.}

\author[Peralta]{Antonio M. Peralta}
\address{Departamento de An{\'a}lisis Matem{\'a}tico, Facultad de
Ciencias, Universidad de Granada, 18071 Granada, Spain.}
\curraddr{Visiting Professor at Department of Mathematics, College of Science, King Saud University, P.O.Box 2455-5, Riyadh-11451, Kingdom of Saudi Arabia.}
\email{aperalta@ugr.es}

\author[Ram\'{i}rez]{Mar{\'\i}a Isabel Ram{\'\i}rez}
\address{Departamento Matem\'aticas, Universidad de
Almer\'ia, 04120 Almer\'ia, Spain} \email{mramirez@ual.es}

\thanks{Authors partially supported by the Spanish Ministry of Science and Innovation,
D.G.I. project no. MTM2011-23843, Junta de Andaluc\'{\i}a grant FQM375. Second author partially supported by the Deanship of Scientific Research at King Saud University (Saudi Arabia) research group no. RG-1435-020. The first author acknowledges the partial financial support from the IEMath-GR program for visits of young talented researchers, and from the Higher Education And Scientific Research In Tunisia project UR11ES52: Analyse, Géométrie et Applications.}


%\date{November 14th, 2014}

\begin{abstract} We prove that every weak-local derivation on a C$^*$-algebra is continuous, and the same conclusion remains valid for weak$^*$-local derivations on von Neumann algebras. We further show that weak-local derivations on C$^*$-algebras and weak$^*$-local derivations on von Neumann algebras are derivations. We also study the connections between bilocal derivations and bilocal $^*$-automorphism with our notions of extreme-strong-local derivations and automorphisms.
\end{abstract}


\maketitle
 \thispagestyle{empty}

\section{Introduction and preliminaries}

A derivation of a Banach algebra $A$ into a Banach $A$-bimodule $X$ is a linear mapping $D: A\to X$ satisfying $$D(a b) = D(a) b + a D(b),$$ for every $a,b\in A$. When $A$ is a C$^*$-algebra, the set Der$(A,X)$ of all derivations of $A$ into $X$ is a closed subspace of the space $B(A,X)$ of all bounded linear operators from $A$ into $X$ (cf. \cite{Ringrose72}). When the set Der$(A,X)$ is regarded as a subspace of $L(A,X)$, the space of all linear maps from $A$ into $X$, it satisfies a strong stability property. Recalling a definition frequently used in the literature (see \cite{KOPR2014} or \cite{BattMol}), we shall say that a subset $\mathcal{D}$ of the Banach space $B(X,Y)$, of all bounded linear operators from $X$ into $Y$, is called \emph{algebraically reflexive} (respectively, \emph{topologically reflexive}) in $B(X,Y)$ when it satisfies the property:
\begin{equation}\label{eq reflexivity} T\in B(X,Y) \hbox{ with } T(x)\in \mathcal{D} (x), \ \forall x\in X \Rightarrow T\in \mathcal{D},
\end{equation}%\vspace{0.01mm}
(respectively,
\begin{equation}\label{eq topological reflexivity} T\in B(X,Y) \hbox{ with } T(x)\in \overline{\mathcal{D} (x)}^{^{\|.\|}}, \ \forall x\in X \Rightarrow T\in \mathcal{D}).
\end{equation}%\vspace{0.01mm}
We shall say that ${\mathcal D}$ is \emph{algebraically reflexivity} (respectively, \emph{topologically reflexive}) in the space $L(X,Y)$, of all linear mappings from $X$ into $Y$, if
\begin{equation}\label{eq reflexivity2} T\in L(X,Y) \hbox{ with } T(x)\in \mathcal{D} (x), \ \forall x\in X \Rightarrow T\in \mathcal{D},
\end{equation}
respectively,
\begin{equation}\label{eq topological reflexivity2} T\in L(X,Y) \hbox{ with } T(x)\in \overline{\mathcal{D} (x)}^{^{\|.\|}}, \ \forall x\in X \Rightarrow T\in \mathcal{D}).
\end{equation}

When in \eqref{eq topological reflexivity} and \eqref{eq topological reflexivity2}, the norm closure of $\mathcal{D} (x)$ is replaced with the closure with respect to another topology $\tau$ on $Y$ we  say that $\mathcal{D}$ is $\tau$-topologically reflexive in $B(X,Y)$ or in $L(X,Y)$.\smallskip

Clearly, $\mathcal{D}$ is topologically reflexive in $B(X,Y)$ or in $L(X,Y)$ whenever it is algebraically reflexive. Some known examples of algebraically and topologically reflexive subsets include the following:\begin{enumerate}[$\checkmark$]
\item The space Der$(M,X)$ of all  derivations of a von Neumann algebra $M$ into a dual $M$-bimodule $X$ is algebraically reflexive in $B(M,X)$ (R.V. Kadison, 1990 \cite{Kad90});
\item For an infinite dimensional separable Hilbert space $H$, the set of automorphisms on the Banach algebra $B(H),$ is algebraically reflexive in $L(H)$ (D.R. Larson, A.R. Sourour, 1990 \cite{LarSou} and M. Bre\v{s}ar and P. \v{S}emrl  \cite[Theorem 2]{BreSemrl95});
\item For a C$^*$-algebra $A$, the space of derivations on $A$ is algebraically reflexive in $B(A)$ (V. Shul'man, 1994 \cite{Shu});
\item For a separable infinite-dimensional complex Hilbert space $H$, the $^*$-automorphism group and the isometry group of the type I factor $B(H)$ are topologically reflexive (C. Batty, L. Moln\'{a}r, 1996 \cite{BattMol});
\item The space Der$(A,X)$ of all  derivations from a C$^*$-algebra $A$ into a Banach $A$-bimodule $X$ is algebraically reflexive in $L(A,X)$ (B.E. Johnson, 2001 \cite{John01});
\item The space $\hbox{Der}_{t} (M)$ of all triple derivations on a JBW$^*$-triple $M$ is algebraically reflexive in $B(M)$ (M. Mackey, 2013 \cite{Mack});
\item The space $\hbox{Der}_{t} (E)$ of all triple derivations on a JB$^*$-triple $E$ is algebraically reflexive in $L(E)$ (M. Burgos, F.J. Fernández-Polo, A.M. Peralta\hyphenation{Peralta}, 2014 \cite{BurFerPe2014}).
\end{enumerate}

In \cite[\S 2]{Pop}, F. Pop shows an example of a local homomorphism from $M_2 (\mathbb{C})$ into $M_4 (\mathbb{C})$ which is not multiplicative (see also \cite[Example 3.13]{Pe2014}). It is also known that the group of automorphisms of a C$^*$-algebra need not be topologically reflexive, an example can be given in $C[0,1]$ (cf. \cite[page 415 and Theorem 5]{BattMol}).\smallskip

In this paper we introduce a property, which is stronger than the property of being algebraically reflexive, and weaker the the property of being topologically reflexive. We shall show that many of the previous examples satisfying the algebraic reflexivity also satisfy the new stronger property.

\begin{definition}\label{def tau-local derivations} Let $X$ and $Y$ be a Banach spaces, and let $\tau$ be a locally convex topology on $Y$ defined by a family of seminorms $\big\{ \||.|\|_i : i\in I\Big\}$. An operator $T\in B(X,Y)$ will be said to be $\tau$-locally in $\mathcal{D}$ if \begin{equation}\label{eq tau-locally in D} \hbox{for each $x\in X$ and each $i\in I$, there exists $D_{x,i}\in \mathcal{D}$,}
\end{equation} $$\hbox{ depending on $x$ and $i$, such that $\||T(x) -D_{x,i} (x)|\|_i =0.$}$$\vspace{0.1mm}

A subset $\mathcal{D}$ of the Banach space $B(X,Y)$, of all bounded linear operators from $X$ into $Y$, will be called \emph{$\tau$-algebraically reflexive} in $B(X,Y)$ when every operator $T\in B(X,Y)$ being $\tau$-locally in $\mathcal{D}$ belongs to $\mathcal{D}$. We can consider a similar definition replacing $B(X,Y)$ with the space $L(X,Y)$, of all linear mappings from $X$ into $Y$.
\end{definition}

It is clear that every algebraically reflexive subset $\mathcal{D}$ in $B(X,Y)$ or in $L(X,Y)$ is $\tau$-algebraically reflexive. And $\mathcal{D}$ being  $\tau$-algebraically reflexive implies that $\mathcal{D}$ is $\tau$-topologically reflexive.\smallskip

Let $A$ be a C$^*$-algebra. The symbol $S(A)$ will denote the set of states on $A$, (i.e. the set of all norm-one, positive functionals in $A^*$). Given a positive functional $\phi\in S(A)$ we consider two seminorms on $A$: $$\||a|\|_{\phi} = \phi (a^* a)^{\frac12},\hbox{ and } |a|_{\phi}:=|\phi (a)| \ (a\in A).$$ We shall pay special attention in the following cases: and let $\tau_1$ and $\tau_2$ be the topologies on $A$ given by the families $\{  |\cdot|_{\phi} : \phi\in S(A)\}$ and $\{\||\cdot|\|_{\phi} : \phi\in S(A) \}$. Clearly, $\tau_1$ is the weak topology of $A$ and $\tau_2$ coincides with the restriction to $A$ of the strong topology of $A^{**}$ (cf. \cite[Definition 1.8.6]{Sak}).\smallskip

Given a von Neumann algebra $M$, with predual $M_*$, we shall write $S_{n} (M)$ for the set of all normal states on $M$ (i.e. the set of all norm-one, positive functionals in $M_*$). We recall \cite[Definitions 1.8.6 and 1.8.7]{Sak} that the \emph{strong topology} (respectively, the \emph{strong$^*$ topology}) of $M$ is the locally convex topology on $M$ defined by the family $\{\||\cdot|\|_{\phi} : \phi\in S_n(M) \}$ (respectively, $\{\||\cdot|\|_{\phi}, \||\cdot|\|^*_{\phi} : \phi\in S_n(M) \}$, where $\||a|\|^*_{\phi}  = \||a^*|\|_{\phi}$, $a\in A$). Following standard notation, the strong and the strong$^*$ topologies of $M$ are denoted by $s(M, M_*)$ and $s^*(M, M_*)$, respectively.\smallskip

When $\emph{D} = \hbox{Der}(A)\subset B(A),$ is the set of all derivations on $A$, a linear map $T: A\to A$ which is $\tau_1$-locally in $\hbox{Der}(A)$ will be called a \emph{weak-local derivation on $A$}. Further, when $\emph{D} = \hbox{$^*$-Aut}(A)\subset B(A),$ is the set of all $^*$-automorphisms of $A$, a linear map $T: A\to A$ which is $\tau_1$-locally in $\hbox{$^*$-Aut}(A)$ will be called a \emph{weak-local $^*$-automorphism on $A$}. We similarly define \emph{strong-local derivations}, \emph{strong-local $^*$-automorphisms}, \emph{weak-local {\rm(}$^*$-{\rm)} homomorphisms}, and \emph{strong-local {\rm(}$^*${\rm)}-homomorphisms} on $A$.\smallskip

Given $\phi \in S_n(M_*)$, where $M$ is a von Neumann algebra, by the Cauchy-Schwartz inequality $$|a|_{\phi}^2= |\phi (a)|^2 \leq \phi (a^* a) \phi (1^* 1) = \|| a |\|_{\phi}^2 \ \ \ {\rm(} a\in M {\rm ).}$$ Therefore every strong-local derivation (respectively, $^*$-automorphism) on a C$^*$-algebra is a weak-local derivation  (respectively, a $^*$-automorphism).\smallskip

Clearly, every local derivation (respectively, every local $^*$-automorphism) on $A$ is a \emph{$\tau$-local derivation} (respectively, a \emph{$\tau$-local $^*$-automorphism}) on $A$ for $\tau= \tau_1$ or $\tau_2.$ In this note we shall prove that every weak-local derivation on a C$^*$-algebra is a derivation, a result which extends the famous theorems of R.V. Kadison \cite{Kad90} and B.E. Johnson \cite{John01}.\smallskip

When $M$ is a von Neumann algebra, we can also consider the topologies $\tau_3$ and $\tau_4$ generated by the families $\{  |\cdot|_{\phi} : \phi\in S_n (M)\}$ and $\{\||\cdot|\|_{\phi} : \phi\in S_n(M) \}$, which clearly coincide with the weak$^*$ and the strong$^*$ topologies of $M$, respectively. We shall also consider \emph{weak$^*$- and strong$^*$-local derivations} and \emph{weak$^*$- and strong$^*$-local $^*$-automorphisms} on $M$. Clearly, every strong-local derivation (respectively, $^*$-automorphism) on a von Neumann algebra is a weak$^*$-local derivation  (respectively, $^*$-automorphism).\smallskip

In Section \ref{sec: automatic cont} we prove that every weak-local derivation on a C$^*$-algebra is continuous (Theorem \ref{thm cont weak local derivations}). Deeper arguments are needed to establish that every weak$^*$-local derivation on a von Neumann algebra is continuous (Theorem \ref{t automatic continuity weak*local derivation}). These results generalize classical results on automatic continuity derivations due to S. Sakai \cite{Sak60}, J.R. Ringrose \cite{Ringrose72}, and B.E. Johnson \cite{John01}. Among the new tolls developed in section, we show that every linear left-annihilator-preserving (respectively, \emph{right-annihilator-preserving}) on a von Neumann algebra is continuous, and hence a left multiplier (see Corollary \ref{c prop automatic cont left-annihilator-preserving on  von Neumann}).\smallskip

The main results established in Section \ref{sec: weak local derivations} prove that the space of derivations on a von Neumann algebra $M$ (respectively, on a C$^*$-algebra $A$) is weak$^*$-algebraically reflexive in $L(M)$ (respectively, weak-algebraically reflexive in $L(A)$) (see Theorems \ref{t weak*-local derivations are derivations} and \ref{t weak-local derivations are derivations}, respectively).\smallskip

Section \ref{sec: strong-local automorphims on von Neumann algebras} is devoted to prove that every strong-local $^*$-automorphism on a von Neumann algebra is a Jordan $^*$-homomorphism (Theorem \ref{t strong$^*$-local *-automorphisms}).\smallskip

The concepts studied in this paper also admits some connections with more recent contributions. In \cite{ZhuXiong97},  C. Xiong and J. Zhu introduced the notion of bilocal derivation on $B(H)$ in the following sense: a linear map $T: B(H)\to B(H)$ is a \emph{bilocal derivation} if for every $a\in B(H),$ and every $\xi\in H$, there exists a derivation $D_{a,\xi} : B(H)\to B(H)$, depending on $a$ and $\xi$, such that $\|T(a) (\xi) - D_{a,\xi} (a) (\xi)\|=0.$ Clearly, we can restrict to the case $\|\xi\|=1.$

Inspired by the above notion we define here extreme-$\tau$-local derivations and automorphisms. A linear mapping $T: M \to M$ is said to be an extreme-weak$^*$-local derivation {\rm(}respectively, an extreme-strong$^*$-local derivation{\rm)} if for every $a\in M$, and every pure normal state $\phi\in \partial_{e} (S_n(M))$, there exists a derivation $D_{a,\phi}: M\to M,$ depending on the elements $a$ and $\phi$, such that $\Big|\phi \Big(T(a) - D_{a,\phi} (a)\Big)\Big|=0,$ {\rm(}respectively, $\||T(a) - D_{a,\phi} (a)|\|_{\phi }=0${\rm)}. Extreme-weak$^*$-local $^*$-automorphism and extreme-strong-local $^*$-automorphism are similarly defined. Bilocal derivations on $B(H)$ are precisely the extreme-strong$^*$-local derivations on $B(H)$ (see Remark \ref{reamrk bilocal derivations are extreme strong* derivations}). We prove here that every (linear) extreme-weak$^*$-local derivation on an atomic von Neumann algebra is continuous (Theorem \ref{t automatic continuity extreme-weak*local derivation}). Improving a result of C. Xiong and J. Zhu \cite[Theorem 3]{ZhuXiong97}, we further show that every extreme-weak$^*$-local derivation on an atomic von Neumann algebra is a derivation (Theorem \ref {t extreme-weak*-local derivations are derivations}).\smallskip

In 2014, L. Moln\'{a}r introduced and studied bilocal $^*$-automorphisms on $B(H)$. Concretely, a linear mapping $T: B(H) \to B(H)$ is said to be a \emph{bilocal $^*$-automorphism} if for every $a$ in $B(H)$ and every $\xi$ in $H$, there exists a $^*$-automorphism $\pi_{a,\xi}: B(H)\to B(H)$, depending on $a$ and $\xi$, such that $T(a) (\xi) = \pi_{a,\xi} (a) (\xi)$ (cf. \cite{Mol2014}). Bilocal $^*$-automorphisms and extreme-strong-local $^*$-automorphisms on  $B(H)$ define the same applications. In \cite[Theorem 1]{Mol2014}, L. Moln\'{a}r establishes that for a linear transformation $T: B(H)\to B(H)$, where $H$ is an infinite dimensional and separable complex Hilbert space, $T$ is a bilocal $^*$-automorphism if and only if $T$ is a unital algebra $^*$-endomorphism. We prove in this note that every extreme-strong-local $^*$-automorphism on an atomic von Neumann algebra is a Jordan $^*$-homomorphism (Theorem \ref{t extreme-strong-local *-automorphisms}). As a consequence, we establish that the conclusion of Molnár's theorem is also valid for arbitrary Hilbert spaces.

\section{Automatic continuity of weak-local derivations}\label{sec: automatic cont}

Throughout this paper, given a Banach space $X$, the symbols $X_1$ and $S_{X}$ will denote closed unit ball and the unit sphere of $X$, respectively. The self-adjoint part of a C$^*$´-algebra $A$ will be denoted by $A_{sa}.$\smallskip

We recall that a linear mapping  $R: X\longrightarrow X$ on Banach space $X$ is said to be \emph{dissipative}, if for every $(x,\;\phi)\in S_X\times S_{X^*}$ with
$\phi(x)=1,$ we have $\Re\hbox{e}(\phi T(x))\leq 0.$ It is known that every dissipative linear map on a Banach space is continuous (compare \cite[Proposition 3.1.15]{BraRo}).\smallskip

Let $A$ and $B$ be C$^*$-algebras, and let $T: A\to B$ be a linear mapping. We define $T^{\sharp} : A\to B$ the linear mapping defined by $T^{\sharp} (a) := T(a^*)^*$ ($a\in A$). We shall say that $T$ is symmetric when $T^{\sharp} = T.$  A derivation $D$ on a C$^*$-algebra $A$ is said to be a $^*$-derivation when it is a derivation and a symmetric map (i.e., $D(a^*) = D(a)^*,$ for every $a\in A$). Let us observe that $D^{\sharp}$ is a derivation whenever $D:A\to A$ is.\smallskip

It is originally due to S. Sakai that every derivation on a C$^*$-algebra is continuous \cite{Sak60}. Some years later, A. Kishimoto proves in \cite[Corollary, page 27]{Kishi76} that every $^*$-derivation on a C$^*$-algebra $A$ is dissipative, and hence continuous (see also \cite[\S 1.4]{Bra1986}). Given a general derivation $D$ on $A$, we can write $D = D_1 + i D_2,$ where $D_1 = \frac12 (D+ D^{\sharp})$ and $D_2 = \frac{1}{2 i} (D- D^{\sharp})$ are $^*$-derivations on $A$. Therefore, the previous result of Kishimoto also assures that every derivation on a C$^*$-algebra is continuous.\smallskip

J.R. Ringrose establishes in \cite{Ringrose72} that every derivation from a C$^*$-algebra $A$ into a Banach $A$-bimodule is continuous, and B.E. Johnson extended the above results in \cite[Theorem 7.5]{John01} by showing that every local derivation of a C$^*$-algebra $A$ into a Banach $A$-bimodule is continuous.\smallskip

Here we consider the automatic continuity of weak-local derivations on a C$^*$-algebra.

\begin{theorem}\label{thm cont weak local derivations}
Every weak-local derivation on a C$^*$-algebra is continuous.
\end{theorem}

\begin{proof} Let $T: A\to A$ be a weak-local derivation on a C$^*$-algebra. Let us write $T = T_1 + i T_2,$ where $T_1 = \frac12 (T+ T^{\sharp})$ and $T_2 = \frac{1}{2 i} (T- T^{\sharp})$. We claim that $T_1$ and $T_2$ are weak-local derivations on $A$. Indeed, by hypothesis, given $a\in A$ and $\phi\in S(A)$ there exist derivations $D_{a,\phi}, D_{a^*,\phi}: A\to A$ such that $\phi T(a) = \phi D_{a,\phi}(a)$ and $\phi T(a^*) = \phi D_{a^*,\phi}(a^*)$. Since $\phi\in S(A),$ we have $\overline{\phi (T(a^*)^*)} = \phi T(a^*) = \phi D_{a^*,\phi}(a^*)$. Therefore $$\phi T_1 (a)  =  \frac12 \left(\phi T(a) + \phi (T(a^*)^*) \right)= \frac12 \left(\phi D_{a,\phi} (a) + \overline{\phi (D_{a^*,\phi} (a^*))} \right) $$ $$= \frac12 \left(\phi D_{a,\phi} (a) + \phi D_{a^*,\phi}^{\sharp} (a) \right) = \phi \left(\frac{D_{a,\phi} +  D_{a^*,\phi}^{\sharp} }{2}\right)  (a),$$ and similarly, $$\phi T_2 (a)  = \phi \left(\frac{D_{a,\phi} -  D_{a^*,\phi}^{\sharp} }{2i}\right)  (a),$$ which proves the claim.\smallskip

The arguments in the above paragraph shows that we can assume $T=T^{\sharp}$ is a symmetric mapping. We shall show that $T|_{A_{sa}} : A_{sa} \to A_{sa}$ is dissipative. Let $(a,\phi)$ in $S_{A_{sa}}\times S_{A_{sa}^*}$ with $\phi (a) = 1$. By hypothesis, there exists a derivation $D_{a,\phi}: A\to A$ such that $\phi T(a) = \phi D_{a,\phi}(a).$ Having in mind that $a=a^*$ and $\phi\in A_{sa}^*$, we deduce that $$\phi D_{a,\phi}^{\sharp} (a) = \phi (D_{a,\phi} (a^*)^*) = \overline{\phi (D_{a,\phi} (a))} .$$ Thus, $$\phi T(a) = \phi D_{a,\phi} (a) = \frac12 (\phi D_{a,\phi} (a) +  \overline{\phi (D_{a,\phi} (a))}) = \phi \left(\frac{D_{a,\phi} + D_{a,\phi}^{\sharp}}{2} \right) (a) \leq 0,$$ because $\frac{D_{a,\phi} + D_{a,\phi}^{\sharp}}{2}$ is a $^*$-derivation on $A$. This shows that $T|_{A_{sa}}$ is dissipative as desired.
\end{proof}

In the case of von Neumann algebra $M$ we can also consider the automatic continuity of a weak$^*$-local derivation on $M$. However this question is more difficult to answer due to the lacking of a result of automatic continuity for \emph{weak$^*$-dissipative maps} on dual Banach spaces. We shall illustrate this statement with the following example: Consider an unbounded linear mapping $T:\ell_1\to \ell_1$ satisfying that $T(e_n)=0,$ for every $n\in \mathbb{N}$, where $\{e_n: n\in \mathbb{N}\}$ is the Schauder basis of $\ell_1$. Let $a\in \ell_1= c_0^*$, $\varphi\in c_0$ be norm-one elements satisfying $\varphi(a) = 1$. Since $a$ is a norm attaining functional in $\ell_1= c_0^*$, it is well known that $a$ must be a finite linear combination of elements in the basis  $\{e_n: n\in \mathbb{N}\}$, and thus, $T(a) = 0$. This implies that $\varphi T(a) =0\leq 0.$ However $T$ is unbounded. \smallskip

We shall establish some technical results now. The next lemma is probably well known in the folklore of von Neumann algebras.

\begin{lemma}\label{l one} Let $p$ be a projection in a von Neumann algebra $M$. Suppose $a$ is an element in $M$ satisfying $\phi (a) = 0$ for every $\phi \in S_n (M)$ with $p\phi p = \phi$ {\rm(}i.e. $\phi (px p) = \phi (x)$ for every $x\in M${\rm)}. Then $p a p =0.$
\end{lemma}

\begin{proof} Let us observe that $pMp$ is a hereditary von Neumann subalgebra of $M$ with predual $(pMp)_* =p M_* p$. The statement follows from \cite[Lemma 1.7.2]{Sak} applied to $pap$ and $pMp.$
\end{proof}

\begin{lemma}\label{l two} Let $T: M\to M$ be a weak$^*$-local derivation on a von Neumann algebra. Then the following statements hold:\begin{enumerate}[$(a)$]\item $T(1) =0;$
\item $T(p) = p T(p) (1-p) + (1-p) T(p) p,$ for every projection $p\in M.$
\end{enumerate}
\end{lemma}

\begin{proof}
$(a)$ Let us recall that for every derivation $D$ on a unital C$^*$-algebra $D(1)=0.$ By assumptions, for each $\phi \in M_*,$ there exists a derivation $D_{1,\phi} :M\to M$ such that $\phi T(1) = \phi D_{1,\phi} (1) =0.$ Thus, the conclusion of $(a)$ follows from \cite[Lemma 1.7.2]{Sak}.\smallskip

$(b)$ Let us fix a projection $p\in M,$ and let us take $\phi\in S_n(M)$ such that $(1-p)\phi (1-p) =\phi$. We claim that $\phi T(p) =0.$ Indeed, by hypothesis, there exists a derivation $D_{p,\phi} :M\to M$ such that $$\phi T(p) = \phi D_{p,\phi} (p) = \phi (D_{p,\phi} (p) p + p D_{p,\phi} (p))$$ $$ =\phi ((1-p) D_{p,\phi} (p) p (1-p) + (1-p) p D_{p,\phi} (p)(1-p)) =0 ,$$ which proves the claim. Applying Lemma \ref{l one}, we deduce that $$(1-p) T(p) (1-p)=0.$$ Replacing $p$ with $1-p$ and having in mind that, by $(a)$, $T(1)=0$, we get $T(p) = p T(p) (1-p) + (1-p) T(p) p.$
\end{proof}

We observe that when $T:M\to M$ is a continuous linear mapping on a von Neumann algebra satisfying the conclusion in Lemma \ref{l two}$(b)$, then $T$ must be a derivation (cf. \cite[Theorem 2]{Bre92} or \cite[Proof of Theorem 2.1]{AyuKudPe2014}). However, we do not know yet whether every weak$^*$-local derivation on a von Neumann algebra is continuous.\smallskip

Let $X$ and $Y$ be Banach $A$-bimodules over a Banach algebra $A$. Given a subset $S\subset X$ the \emph{left-annihilator} (respectively, the \emph{right-annihilator}) of $S$ in $A$ is the set $$\hbox{Ann}_{l,A} (S)= \hbox{Ann}_{l} (S)  := \left\{ a\in A : a S = 0 \right\},$$ (respectively, $\hbox{Ann}_{r,A} (S)= \hbox{Ann}_{r} (S)  := \left\{ a\in A : S a = 0 \right\}$).\smallskip

We recall that a mapping $f: X\to Y$ is said to be a \emph{left-annihilator-preserving} (respectively, \emph{right-annihilator-preserving}) if $f(x) a=0$, whenever $x a=0$ (respectively, $a f(x)=0$, whenever $ a x=0$) with $a\in A$, $x\in X$. A linear map $T: A\to X$ is called a \emph{left (respectively, right) multiplier} if $T (ab) = T (a)b$ (respectively, $T (ab) = a T (b)$), for every $a,b\in A$. Clearly, every left (respectively, right) multiplier is a left-annihilator-preserving (respectively, a right-annihilator-preserving) mapping. J. Lin and Z. Pan proved in \cite[Theorem 2.8]{LiPan} that every bounded and linear left-annihilator-preserving (respectively, every bounded and linear right-annihilator-preserving) mapping from a unital C$^*$-algebra $A$ into a unital Banach $A$-bimodule is a  left multiplier (respectively, a right multiplier). Lemma 2.10 in \cite{AyuKudPe2014} shows that the same conclusion remains valid for bounded and linear left-annihilator-preserving (respectively, every bounded and linear right-annihilator-preserving) maps from a C$^*$-algebra $A$ into an essential Banach $A$-bimodule.\smallskip

In the conditions above, a linear mapping $L:A \to X$ is called a \emph{local left}  (respectively, \emph{right}) \emph{multiplier} if for every $a\in A$ there exists a left  (respectively, right) multiplier $T_{a}: A\to X$, depending on $a$, such that $L(a) = T_a (a)$. B.E. Johnson proved in \cite[Proposition 7.2]{John01} that in the case $A= C_0 (L)$, where $L$ is a locally compact topological space, every local multiplier, not assumed a priori to be continuous, from $A$ into a left Banach $A$-module is continuous. Our next result is a strengthened version of Johnson's result.

\begin{proposition}\label{p Johnson continuity left-annihilator-preserving on abelian von Neumann} Let $\mathcal{B}$ be a commutative von Neumann subalgebra of a von Neumann algebra $M$, and suppose that $\mathcal{B}$ contains the unit of $M$. Then every linear left-annihilator-preserving {\rm(}respectively, \emph{right-annihilator-preserving}{\rm)} $T: \mathcal{B} \to M$ is continuous.
\end{proposition}

Before proving the proposition we state a technical lemma inspired from \cite[Lemma 9.1]{Sinc76}.\smallskip

Let $(p_i)_i$ be a family of mutually orthogonal projections in a von Neumann algebra $M$. By \cite[page 30]{Sak}, the family $(p_i)_i$ is summable with respect to the weak$^*$-topology (and also with respect to the strong$^*$-topology) of $M,$ and $\displaystyle \sum_{i} p_i$ is another projection in $M$. A similar argument also shows that when $(x_i)_i$ is a bounded family of mutually orthogonal elements in $M,$ then it is summable with respect to the weak$^*$-topology of $M$ (and hence with respect to the strong$^*$-topology of $M$), and the limit  $\displaystyle \sum_{i} x_i$ is another element in $M.$

\begin{lemma}\label{l infinite sequence of m orthogo projections} Let $\mathcal{B}$ be a commutative von Neumann subalgebra of a von Neumann algebra $M$, and suppose that $\mathcal{B}$ contains the unit of $M$. Suppose $(p_n)$ is a sequence of non-zero mutually orthogonal projections in $\mathcal{B}$, and  $T: \mathcal{B} \to M$ is a linear left-annihilator-preserving {\rm(}respectively, \emph{right-annihilator-preserving}{\rm)}. Then $p_n T: \mathcal{B} \to M$, $a\mapsto p_n T(a)$ {\rm(}respectively, $ Tp_n: \mathcal{B} \to M$, $a\mapsto T(a) p_n ${\rm)} is continuous for all but a finite number of $n$.
\end{lemma}

\begin{proof} Let us take a sequence $(p_n)$ of non-zero mutually orthogonal projections in $\mathcal{B}$. Arguing by contradiction we can assume that $p_n T: \mathcal{B} \to M $ is unbounded for every $n.$ So, for each natural $n$, we can find a norm-one element $x_n$ in $\mathcal{B}$ such that $\|p_n T(x_n)\| > 4^{2n}.$ Let us observe that, since $T$ is a  linear left-annihilator-preserving, we have $$p_n T(x ) = p_n T( p_n x + (1-p_n) x) =  p_n T( p_n x )  + p_n T( (1-p_n) x) = p_n T(p_n x),$$ and $$T(p_n x) = p_n T(p_n x) + (1-p_n ) T(p_n x) = p_n T(p_n x),$$ which proves that \begin{equation}\label{eq trick for left annihilator preserving} p_n T(x ) = T(p_n  x ),
 \end{equation}for every $n$ and every $x\in \mathcal{B}$. Thus, $$\| T (p_n x_n ) \| = \|p_n T(x_n)\| > 4^{2n},$$ for every natural $n.$\smallskip

Since the elements in the sequence $(p_n x_n)$ are mutually orthogonal because $\mathcal{B}$ is commutative,  and $\|p_n x_n\|\leq 1$ for every $n,$ the series $\displaystyle z= \sum_{n=1}^{\infty} p_n x_n$ is weak$^*$- and strong$^*$- summable in the von Neumann algebra $\mathcal{B}$. In this case,  $$T(z) = T(p_m x_m ) + T\left( \sum_{n=1, n\neq m}^{\infty} p_n x_n \right) ,$$ which proves that $p_m T(z) =  T(p_m x_m ),$ for every natural $m$. Therefore, $$\| T(z) \| \geq \|p_m T(z) \| = \|T(p_m x_m)\| > 4^{2m},$$ for every natural $m,$ which is impossible.
\end{proof}

Let $T: X\to Y$ be a linear mapping between two normed spaces. Following \cite[page 7]{Sinc76}, the
\emph{separating space}, $\sigma_{_Y} (T)$, of $T$ in $Y$ is defined as the set of all $z$ in $Y$ for which there exists a sequence $(x_n) \subseteq X$ with $x_n \rightarrow 0$ and $T(x_n)\rightarrow z$. An application of the Closed Graph theorem shows that a linear mapping $T$ between two Banach spaces $X$ and $Y$ is continuous if and only if $\sigma_{_Y} (T) =\{0\}$.  It is also known that $\sigma_{_Y} (T)$ is a closed linear subspace of $Y.$ Consequently, for each bounded linear operator $R$ from $Y$ to another Banach space $Z$, the composition $R T$ is continuous if, and only if, $\sigma_{_Y}(T)\subseteq \ker (R)$.\smallskip

\begin{proof}[Proof of Proposition \ref{p Johnson continuity left-annihilator-preserving on abelian von Neumann}] Let us define $I := \hbox{Ann}_{l,B} (\sigma_{M} (T))$. It is easy to check that, since $\mathcal{B}$ is abelian, $I$ is a norm-closed ideal of $\mathcal{B}$. Furthermore, by the separate weak$^*$-continuity of the products in $\mathcal{B}$ and $M$, $I$ is weak$^*$-closed too.\smallskip

We claim that \begin{equation}\label{eq I is bounded products} I = \left\{ a\in \mathcal{B}: a T : \mathcal{B} \to M \hbox{ is continuous }\right\}.
 \end{equation}Indeed, for every element $a\in I$ we have $\sigma_{M} (T) \subseteq \ker (L_a)$, where $L_a : M \to M,$ $L_a (x) =a x$. So, the composition $L_a \circ T = a T : \mathcal{B}\to M$ is continuous. On the other hand, if the mapping $aT$ is continuous, for every element $b\in \sigma_{M} (T),$ there exists a sequence $a_n\to 0$ in norm such that $\|T(a_n) -b\|\to 0$. The continuity of $a T$ shows that $0=a T(a_n) \to a b$ in norm, which shows that $a b=0,$ witnessing the desired equality.\smallskip

Since $I$ is a weak$^*$-closed ideal of $\mathcal{B}$, we known that taking $J = I^{\perp} := \{c\in \mathcal{B} : c I = 0\}$, then $J$ is a weak$^*$-closed ideal in $\mathcal{B}$ and $\mathcal{B} = I\oplus J$. By hypothesis, $\mathcal{B}$ is a commutative von Neumann subalgebra of $M$,  so, it is well known that $\mathcal{B}$ is isometrically isomorphic to some $C(\Omega)$, where $\Omega$ is an Stonean space (\cite[Lemma 1.7.5]{Sak}). It is part of the folklore in Banach algebra theory that, in this case, there exists a clopen subset $\Gamma\subset \Omega$ such that $\Omega = \Gamma \stackrel{\circ}{\bigcup} (\Omega\backslash \Gamma),$ $I = \{b \in C(\Omega) : b|_{\Gamma} =0\} = C(\Omega\backslash \Gamma)$, and $J = \{b \in C(\Omega) : b|_{\Omega\backslash\Gamma} =0\} = C(\Gamma)$ (cf. \cite[Example 2.1.9]{BraRo}). We observe that $\Gamma$ and $\Omega\backslash \Gamma$ both are Stonean spaces.\smallskip

We claim that $\Gamma$ is finite. Otherwise, we can find a sequence $(p_n)$ of non-zero mutually orthogonal projections in the infinite-dimensional commutative von Neumann algebra $J= C(\Gamma).$ We note that, by \eqref{eq I is bounded products}, for each $a\in J\backslash\{0\}$, the mapping $a T : \mathcal{B} \to M$ is unbounded. Thus, $p_n T$ is unbounded for every natural $n,$ which contradicts Lemma \ref{l infinite sequence of m orthogo projections}. We have therefore shown that $\Gamma = \{t_1,\ldots,t_n\}$ is a finite set of isolated points in $\Omega$. Since $J$ is finite dimensional, $T|_{J} : J \to M$ must be continuous.\smallskip

Let $u_{I}$ and $u_{J}$ denote the unit elements in $I$ and $J$ respectively. Since $u_{I}\in I$, the mapping $u_I T $ is continuous. Furthermore, the arguments in the proof of Lemma \ref{l infinite sequence of m orthogo projections} \ref{eq trick for left annihilator preserving}, show that $u_{J} T(x) = T (u_{J} x) = T|_{J} (u_{J} x)$, for every $x\in \mathcal{B}$, therefore $u_{J} T$ is continuous, and hence $u_J\in I$. which proves that $J=\{0\}$ and $T =  u_{I} T$ is continuous.
\end{proof}

In \cite[Theorem 1.3]{Cuntz}, J. Cuntz proved a conjecture posed by J. R. Ringrose in \cite{Ringrose74}, showing that if $A$ is a C$^*$-algebra and $T$ is a linear mapping from $A$ into a Banach space $X$ such that the restriction of $T$ to the C$^*$-subalgebra of $A$ generated by a single hermitian element $h$ in $A$ is continuous, $T$ is bounded on the whole of $A$. A similar statement was established by Ringrose when $A$ is a von Neumann algebra in \cite{Ringrose74}. The following result is a direct consequence of the above Proposition \ref{p Johnson continuity left-annihilator-preserving on abelian von Neumann}, Cuntz theorem and \cite[Theorem 2.8]{LiPan} (see also \cite[Lemma 2.10]{AyuKudPe2014}).

\begin{corollary}\label{c prop automatic cont left-annihilator-preserving on  von Neumann} Every linear left-annihilator-preserving (respectively, \emph{right-annihilator-preserving}) on a von Neumann algebra is continuous, and hence a left multiplier. $\hfill\Box$
\end{corollary}

We would like to note that the above corollary guarantees that the continuity hypothesis in \cite[Theorem 2.8]{LiPan} and \cite[Lemma 2.10]{AyuKudPe2014} can be omitted in the setting in which $X$ and $A$ both coincide with a von Neumann algebra $M$. Corollary \ref{c prop automatic cont left-annihilator-preserving on  von Neumann} and Proposition \ref{p Johnson continuity left-annihilator-preserving on abelian von Neumann} also show that the conclusion of \cite[Proposition 7.2]{John01} also holds when $T: \mathcal{B}\to M$ is a linear left-annihilator-preserving instead of a local multiplier, where in this case $\mathcal{B}$ is a commutative von Neumann subalgebra of $M$ containing the unit in the latter algebra.\smallskip

Given a self-adjoint element $a$ in a von Neumann algebra $M.$ The \emph{range or support projection} of $a$ in $M$ is the smallest projection $p\in M$ satisfying $a p = p a = a$ (compare \cite[Definition 1.10.3]{Sak} and \cite[2.2.7]{Ped}). The range projection of $a$ in $M$ will be denoted by $r(a).$ It is known that the sequence $(a^{\frac1n})$ converges to $r(a)$ with respect to the strong$^*$-topology of $M$.\smallskip

We can establish now the second main result of this section. First we state an strengthened version of Lemma \ref{l two}.

\begin{lemma}\label{l two bis} Let $T: M\to M$ be a weak$^*$-local derivation on a von Neumann algebra. Suppose $a$ is a self-adjoint element in $M$ with range projection denoted by $r(a)$. Then $ (1-r(a)) T(a) (1-r(a)) = 0.$
\end{lemma}

\begin{proof}
Let us take $\phi\in S_n(M)$ such that $(1-r(a))\phi (1-r(a)) =\phi$. We claim that $\phi T(a) =0.$ Indeed, by hypothesis, there exists a derivation $D_{a,\phi} :M\to M$ such that $$\phi T(a) = \phi D_{a,\phi} (a)= \phi D_{a,\phi} (r(a) a) = \phi (D_{a,\phi} (r(a)) a + r(a) D_{a,\phi} (a))$$ $$ =\phi ((1-r(a)) D_{a,\phi} (r(a)) a (1-r(a)) + (1-r(a)) r(a) D_{a,\phi} (a)(1-r(a))) =0 ,$$ which proves the claim. Applying Lemma \ref{l one}, we deduce that $$(1-r(a)) T(a) (1-r(a)) = 0.$$
\end{proof}

\begin{theorem}\label{t automatic continuity weak*local derivation} Every (linear) weak$^*$-local derivation on a von Neumann algebra is continuous.
\end{theorem}

\begin{proof}
Let $T: M\to M$ be a linear weak$^*$-local derivation on a von Neumann algebra. Let $\mathcal{B}$ denote a commutative von Neumann subalgebra of $M$ containing the unit of the latter algebra. Suppose $a,b$ and $c$ are elements in $\mathcal{B}$ with $a b = bc =0.$ We claim that \begin{equation}\label{eq orthogonal outsiders} a T(b) c =0.
\end{equation} Indeed, let $r(a)$, $r(b)$ and $r(c)$ denote the range projections of $a,b$ and $c$, respectively. Since $\mathcal{B}$ is commutative and $ab=0=bc$, we have $r(a)\leq 1-r(b) $ and $r(c)\leq 1-r(b).$ Lemma \ref{l two bis}, applied to $b+b^*$ and $i (b-b^*)$, and having in mind that $r(b+b^*) \leq r(b)$, and $r(i(b-b^*)) \leq r(b)$, we get $(1-r(b)) T(b+b^*) (1-r(b))=0 = (1-r(b)) T(b-b^*) (1-r(b)),$ so $(1-r(b)) T(b) (1-r(b))=0.$ Therefore, $$a T(b) c = a r(a) T(b) r(c) c = a (1-r(b)) T(b) (1-r(b)) r(c) c =0.$$

We can reproduce now part of the arguments in \cite[Theorem 2.8]{LiPan} and \cite[Lemma 2.10]{AyuKudPe2014}. Fix $a,b\in \mathcal{B}$ and define a linear mapping $L_{a,b} : \mathcal{B}\to M$, $L_{a,b} (x) = aT(bx).$ If we take $c,d\in \mathcal{B}$ with $cd=0$, by \eqref{eq orthogonal outsiders}, $L(c) d = aT(bc)d =0.$ Therefore, $L_{a,b}$ is a linear left-annihilator preserving, and hence, Proposition \ref{p Johnson continuity left-annihilator-preserving on abelian von Neumann} asserts that $L_{a,b}$ is continuous and a left-multiplier, that is, $L_{a,b} (x) = L_{a,b} (1) x,$ for every $x\in \mathcal{B}$.\smallskip

Let us fix $x\in \mathcal{B}$. Defining $R_x : \mathcal{B}\to M$, $R_{x} (z) = T(xz)-T(z)x$ is a linear mapping satisfying that $a R_{x} (b) = 0,$ for every $ab=0$ in $\mathcal{B}$. Therefore, $R_{x}$ is a linera right-annihilator preserving, and by Proposition \ref{p Johnson continuity left-annihilator-preserving on abelian von Neumann}, $R_{x}$ is a continuous right multiplier. Therefore, $$T(y x) -T(y) x = R_x (y) = y R_x (1) = y T(x),$$ for every $x,y\in \mathcal{B}$.\smallskip

We have therefore shown that $T|_{\mathcal{B}}: \mathcal{B}\to M$ is a derivation, whenever $\mathcal{B}$ is an abelian von Neumann subalgebra of $M$ containing the unit element of the latter algebra.  If we regard $M$ as a Banach $\mathcal{B}$-bimodule, Ringrose proves in \cite[Theorem 2]{Ringrose72} that $T|_{\mathcal{B}}$ is a bounded linear map. This shows that the restriction of $T$ to each maximal abelian $^*$-subalgebra $\mathcal{B}$ of $M$ is bounded. Finally, an application of \cite[Theorem 2.5]{Ringrose74} (see also \cite{Cuntz} for completeness) proves that $T$ is bounded on the whole of $M$.\end{proof}

In some particular cases, we can consider a weaker hypothesis than the one assumed in Theorem \ref{t automatic continuity weak*local derivation}. We recall that a von Neumann algebra $M$ is said to be \emph{atomic} if $M$ is C$^*$-isomorphic to a $\ell_{\infty}$-sum of von Neumann algebras of the form $B(H_i),$ where each $H_i$ is a complex Hilbert space. We remark that a von Neumann algebra $M$ is atomic (i.e. $\displaystyle M= \bigoplus^{\infty}_{i} B(H_i)$) if and only if $M$ coincides with the bidual of the compact C$^*$-algebra $\displaystyle A= \bigoplus^{c_0}_{i} K(H_i)$, $K(H_i)$ denotes the space of compact linear operators on $H_i$ (cf. \cite[\S 1.19]{Sak}).\smallskip

Given a von Neumann algebra $M$, we shall denote by $\partial_{e} (S_n(M))$ the pure normal states of $M$, that is the set of all extreme points in $S_n(M)$. We note that in general, $\partial_{e} (S_n(M))$ may be empty. When $M= A^{**}$, the Krein-Milman theorem asserts that $\partial_{e} (S_n(M))$ is non-empty and $\sigma (A^*,A)$-dense in $\left(M_*^{+}\right)_1$, however, it could happen, even in the commutative setting, that $\partial_{e} (S_n(M))$ does not separate the points in $M$. However, when $M$ is atomic, the pure normal states on $M$ separate the points in $M$.
Consequently, our next definition is only useful in the setting of atomic von Neumann algebras.

\begin{definition}\label{def extreme weak$^*$-local derivation} Let $M$ be von Neumann algebra. A linear mapping $T: M \to M$ is said to be an extreme-weak$^*$-local derivation {\rm(}respectively, an extreme-strong$^*$-local derivation{\rm)} if for every $a\in M$, and every pure normal state $\phi\in \partial_{e} (S_n(M))$, there exists a derivation $D_{a,\phi}: M\to M,$ depending on the elements $a$ and $\phi$, such that $\Big|\phi \Big(T(a) - D_{a,\phi} (a)\Big)\Big|=0,$ {\rm(}respectively, $\||T(a) - D_{a,\phi} (a)|\|_{\phi }=0${\rm)}.
\end{definition}

\begin{remark}\label{reamrk bilocal derivations are extreme strong* derivations} In \cite{ZhuXiong97},  C. Xiong and J. Zhu introduced the notion of bilocal derivation on $B(H)$. According to the their terminology, a linear map $T: B(H)\to B(H)$ is a \emph{bilocal derivation} if for every $a\in B(H),$ and every $\xi\in H$, there exists a derivation $D_{a,\xi} : B(H)\to B(H)$, depending on $a$ and $\xi$, such that $\|T(a) (\xi) - D_{a,\xi} (a) (\xi)\|=0.$ Clearly, we can restrict to the case $\|\xi\|=1.$\smallskip

Having in mind that $B(H)$ is an atomic von Neumann algebra, and the functional $\xi\otimes \xi : B(H)\to \mathbb{C},$ $a\mapsto \langle a(\xi), \xi \rangle$ is a pure normal state of $B(H)$, with $\|| a|\|_{\xi\otimes \xi} = (\xi\otimes \xi) (a^* a)^{\frac12} = \|a(\xi)\|$, we see that bilocal derivations on $B(H)$ are precisely the extreme-strong$^*$-local derivations on $B(H)$.
\end{remark}

\begin{theorem}\label{t automatic continuity extreme-weak*local derivation} Every (linear) extreme-weak$^*$-local derivation on an atomic von Neumann algebra is continuous. In particular, every bilocal derivation on $B(H)$ is continuous.
\end{theorem}

\begin{proof} Let $T:M\to M$ be an extreme-weak$^*$-local derivation on an atomic von Neumann algebra. Let us observe that the pure normal states on $M$ separate the points in $M$, so the conclusions of Lemma \ref{l two} remain true for any extreme-weak$^*$-local derivation on an atomic von Neumann algebra. In particular, for every projection $p\in M$, we have $T(p) = (1-p) T(p) p + p T(p) (1-p).$ In a similar fashion, we can prove that $(1-r(a)) T(a) (1-r(a)) = 0,$ for every $a\in M_{sa}$ (i.e. Lemma \ref{l two bis} also holds for $T$). We can therefore reproduce the proof of Theorem \ref{t automatic continuity weak*local derivation} to show that $T$ is continuous.\end{proof}

\begin{remark}\label{remark counterexample in the non atomic setting} Let $A$ be a general C$^*$-algebra satisfying that $A^{**}=M$ is non-atomic. We can decompose $M$ as a direct sum of its atomic part $M_1$ and its non-atomic part $M_2\neq 0$, which satisfies that $\phi|_{M_2} =0$, for every $\phi \in \partial_{e} (S_n(M)).$ We can also assume that $M_2$ is infinite dimensional. Take a derivation $D : M_1\to M_1$ and an unbounded linear mapping $T_2 : M_2\to M_2$ and consider the mapping $T: M\to M$ defined by $T(m_1+m_2):= D (m_1) + T_2 (m_2)$. Clearly, for each $\phi\in \partial_{e} (S_n(M)),$ we have $\phi T(m_1 +m_2) =\phi D(m_1)= \phi \widetilde{D} (m_1+m_2)$, where $\widetilde{D}: M\to M,$ $\widetilde{D} (m_1+m_2) = D(m_1)$ is a derivation on $M.$ Therefore, $T$ is an unbounded extreme-weak$^*$-local derivation, which shows that Theorem \ref{t automatic continuity extreme-weak*local derivation} doesn't hold for non-atomic von Neumann algebras.
\end{remark}

\section{Weak-local derivations}\label{sec: weak local derivations}

We have already commented that a continuous linear operator $T$ on a von Neumann algebra $M$ satisfying that $T(p) = p T(p) (1-p) + (1-p) T(p) p$, for every projection $p$ in $M$ is a derivation (compare \cite[Theorem 2]{Bre92} or \cite[Theorem 2.1]{AyuKudPe2014}). Combining this observation with Lemma \ref{l two}$(b)$ and Theorem \ref{t automatic continuity weak*local derivation} we obtain:

\begin{theorem}\label{t weak*-local derivations are derivations} Every weak$^*$-local derivation on a von Neumann algebra is a derivation. That is, the space of derivations on a von Neumann algebra is weak$^*$-algebraically reflexive. $\hfill\Box$
\end{theorem}

In the proof of Theorem \ref{t automatic continuity extreme-weak*local derivation}, we have shown that Lemma \ref{l two}$(b)$ is also true for any extreme-weak$^*$-local derivation on an atomic von Neumann algebra. Consequently we have:

\begin{theorem}\label{t extreme-weak*-local derivations are derivations} Every extreme-weak$^*$-local derivation on an atomic von Neumann algebra is a derivation.$\hfill\Box$
\end{theorem}

In \cite[Theorem 3]{ZhuXiong97}, C. Xiong and J. Zhu prove that every bilocal derivation on $B(H)$ is a derivation. We have already seen in Remark \ref{reamrk bilocal derivations are extreme strong* derivations} that every bilocal derivation on $B(H)$ is an extreme-strong$^*$-local derivation, and hence a extreme-weak$^*$-local derivation on $B(H)$, so the result by Zhu and Xiong is a consequence of the above Theorem \ref{t extreme-weak*-local derivations are derivations}. Our theorem also shows that extreme-weak$^*$-local derivations, extreme-strong$^*$-local derivations (bilocal derivations) and derivations define the same linear operators on an atomic von Neumann algebra. The following result summarizes these ideas and generalizes \cite[Theorem 3]{ZhuXiong97}. \smallskip

\begin{corollary}\label{cor Xiong Zhu improved} Let $T: B(H) \to B(H)$ be a linear mapping, where $H$ is a complex Hilbert space. The following are equivalent:\begin{enumerate}[$(a)$] \item $T$ is a derivation;
\item $T$ is a local derivation;
\item $T$ is a extreme-strong$^*$-local derivation or a bilocal derivation {\rm(}equivalently, for each $a\in B(H)$ and each $\xi$  in $H$, there exists a derivation $D_{a,\xi}: B(H) \to B(H),$ depending on $a$ and $\xi$, such that $T(a) (\xi) = D_{a,\xi}(a) (\xi)${\rm)};
\item $T$ is a extreme-weak$^*$-local derivation {\rm(}equivalently, for each $a\in B(H)$ and each $\xi$  in $H$, there exists a derivation $D_{a,\xi}: B(H) \to B(H),$ depending on $a$ and $\xi$, such that $\langle T(a) (\xi) | \xi \rangle = \langle D_{a,\xi}(a) (\xi) | \xi\rangle${\rm)}. $\hfill\Box$
\end{enumerate}
\end{corollary}

Theorems \ref{t weak*-local derivations are derivations} and \ref{t extreme-weak*-local derivations are derivations} are generalized Kadison-Johnson type theorems for von Neumann algebras. The main goal of this section is another Kadison-Johnson type theorem for C$^*$-algebras, which extends the above Theorem  \ref{t weak*-local derivations are derivations} to the setting of weak-local derivations on general C$^*$-algebras. The concrete result is the following:

\begin{theorem}\label{t weak-local derivations are derivations} Every weak-local derivation on a C$^*$-algebra is a derivation.  That is, the space of derivations on a C$^*$-algebra is weak-algebraically reflexive.
\end{theorem}

As in previous cases, the proof will rely on a series of lemmas and propositions. We begin with an easy consequence of Lemma \ref{l one}.

\begin{lemma}\label{l one C*-algebras} Let $A$ be a C$^*$-algebra. Suppose $a$ is an element in $A$, $p$ is a projection in $A^{**}$ such that for every $\phi \in S_n (A^{**})$ with $p\phi p = \phi$ we have $\phi (a) =0$. Then $p a p =0.$ $\hfill\Box$
\end{lemma}

\begin{proposition}\label{p range projections in weak-local projections} Let $T: A\to A$ be a weak-local derivation on a C$^*$-algebra. Let $a$ be a self-adjoint element in $A$. Then the identity $$(1-r(a)) T (b) (1-r(a)) =0,$$ holds for every element $b$ in the C$^*$-subalgebra of $A$ generated by $a$. Consequently, $(1-r(a)) T^{**} (r(a)) (1-r(a)) =0$, where $r(a)$ denotes the range projection of $a$ in $A^{**}.$
\end{proposition}

\begin{proof} Theorem \ref{thm cont weak local derivations} assures that $T$ is continuous. We assume first that $0\leq a$. Let $\phi$ be an element in $S_{n} (A^{**}) $ satisfying $(1-r(a)) \phi (1-r(a))= \phi.$ By hypothesis, we can find a derivation $D_{\phi,a} : A\to A$ such that $\phi T(a) = \phi D_{\phi,a} (a) .$ Let $b= a^{\frac12}.$ Having in ming that $r(a) = r(b),$ we have $$\phi T(a) = \phi D_{\phi,a} (a)  =  \phi D_{\phi,a} (b^2) = \phi \Big(D_{\phi,a} (b) \ b + b\  D_{\phi ,a} (b)\Big) $$ $$ = \phi \Big((1-r(a)) \ D_{\phi,a} (b)\  b \ (1-r(a))+ (1-r(a))\  b\ D_{\phi ,a} (b)\ (1-r(a))\Big) = 0.$$ Applying Lemma \ref{l one C*-algebras} we obtain $(1-r(a)) T (a) (1-r(a)) =0.$ Replacing $a$ with $a^{n}$, and observing that $r(a^{n}) = r(a),$ we deduce that $$(1-r(a)) T (a^{n}) (1-r(a)) =0,$$ for every natural $n$. The continuity and linearity of $T$ prove that $$(1-r(a)) T (b) (1-r(a)) =0,$$ for every element $b$ in the C$^*$-subalgebra of $A$ generated by $a$.\smallskip

Suppose now that $a$ is a self-adjoint element in  $A$. Let us write $a= a^{+}- a^{-}$, where $0\leq a^{+},a^{-}$ and $a^{+}\perp a^{-}$. We observe that $r(a) = r(a^{+})+r(a^{-})$, with $r(a^{+})\perp r(a^{-})$. We have shown in the first paragraph that $$(1-r(a^{+})) T (b_1) (1-r(a^{+})) =0 = (1-r(a^{-})) T (b_2) (1-r(a^{-})),$$ for every element $b_1$ in the C$^*$-subalgebra of $A$ generated by $a^{+}$ and every $b_2$ in the C$^*$-subalgebra of $A$ generated by $a^{-}$. In particular, $$ (1-r(a)) T (b_1+b_2) (1-r(a))  $$ $$= (1-r(a^{+})-r(a^{-})) T (b_1+b_2) (1-r(a^{+})-r(a^{-})) =0$$
%$$ =  (1-r(a^{+})-r(a^{-})) T (b_1) (1-r(a^{+})-r(a^{-})) +  (1-r(a^{+})-r(a^{-})) T (b_2) (1-r(a^{+})-r(a^{-})) =0,$$
for every $b_1$ and $b_2$ as above. Since every element $b$ in the C$^*$-subalgebra of $A$ generated by $a$ can be approximated in norm by elements $b_1+b_2$ as above, the statement of the proposition follows from the continuity of $T$. \smallskip

Finally, since $T^{**}$ is weak$^*$-continuous and $r(a)$ lies in the weak$^*$-closure of the subalgebra generated by $a,$ it follows that $(1-r(a)) T^{**} (r(a)) (1-r(a)) =0$, as desired.\end{proof}

Let $D: A\to A$ be a derivation on a C$^*$-algebra. We recall that, by the separate weak$^*$-continuity of the triple product of $A^{**}$, together with the weak$^*$-density of $A$ in $A^{**}$, the mapping $D^{**}: A^{**}\to A^{**}$ also is a derivation. We have already observed that, for each projection $p\in A^{**}$, $p D^{**} (p) p =0 $ (compare Lemma \ref{l two}). Let $x$ and $y$ be positive elements in $A$ with $y^2 = x$. Then $D(x) = D(y^2) = D(y) y + y D(y)$, and thus, \begin{equation}\label{eq new range outer derivation}
r(x) D(x) r(x) = r(x)\ D(y)\ r(x)\ y\ r(x)+ r(x)\ y\ r(x)\ D(y)\ r(x) = 0,
\end{equation} where in the penultimate identity we have applied that $r(x) = r(y)$.

\begin{proposition}\label{p range projections outer in weak-local projections} Let $T: A\to A$ be a weak-local derivation on a C$^*$-algebra. Let $a$ be a self-adjoint element in $A$. Then the identity $r(a) T (b) r(a) =0$ holds for every element $b$ in the C$^*$-subalgebra of $A$ generated by $a$. In particular, $a T(a) a = 0$, for every $a\in A_{sa}$.
\end{proposition}

\begin{proof} We deduce from Theorem \ref{thm cont weak local derivations} that $T$ is continuous. Let us suppose that $0\leq a$. We pick $\phi$ in $S_{n} (A^{**}) $ satisfying $r(a) \phi r(a)= \phi.$ By hypothesis, we can find a derivation $D_{\phi,a} : A\to A$ such that $\phi T(a) = \phi D_{\phi,a} (a) .$ Let $b= a^{\frac12}.$ Since $r(a) = r(b),$ we deduce from \eqref{eq new range outer derivation} that $$\phi T(a) = \phi D_{\phi,a} (a)  =  \phi D_{\phi,a} (b^2) = \phi \Big(D_{\phi,a} (b) \ b + b\  D_{\phi ,a} (b)\Big)  $$ $$ = \phi \Big(r(a) \ D_{\phi,a} (b)\  b \ r(a)+ r(a)\  b\ D_{\phi ,a} (b)\ r(a)\Big) $$  $$ = \phi \Big(r(b) \ D_{\phi,a} (b)\ r(b)\ b \ r(b)+ r(b)\  b\ r(b)\ D_{\phi ,a} (b)\ r(b)\Big) = 0.$$ Lemma \ref{l one C*-algebras} implies that  $r(a) T (a) r(a) =0.$ The rest is clear.
\end{proof}

\begin{proof}[Proof of Theorem \ref{t weak-local derivations are derivations}] Let $T: A\to A$ be a weak-local derivation on a C$^*$-algebra. It follows from Theorem \ref{thm cont weak local derivations} that $T$ is continuous. Let $a,b$ and $c$ be a self-adjoint elements in the closed unit ball of $A$ with $a b= bc=0$. Clearly, $r(a) \perp r(b)$ and $r(c) \perp r(b)$. Applying Proposition \ref{p range projections in weak-local projections} we get: \begin{equation}\label{equation outers orthogonal} a T(b) c = a r(a) (1-r(b)) T(b) (1-r(b)) r(c) c = 0,
\end{equation} for every $a,b$ and $c$ in $A_{sa}$ with $ab= bc=0$. Theorem 2.10 ($(i^\prime) \Rightarrow (a)$) in \cite{AyuKudPe2014} implies that $T^{**} : A^{**} \to A^{**}$ is a generalized derivation, that is, \begin{equation}\label{eq generalized derivation} T^{**}(x y) = T^{**}(x) y + x T^{**}(y) - x T^{**} (1) y,
\end{equation} for every $x,y\in A^{**}$.\smallskip

Finally, Proposition \ref{p range projections outer in weak-local projections} tells that $r(a) T(a) r(a) =0 = a T(a) a,$ for every $a\in A_{sa}$. By the Kaplansky's density theorem we know that the closed unit ball of $A_{sa}$ is strong$^{**}$ dense in the closed unit ball of $A^{**}_{sa}$ (compare \cite[Theorem 1.9.1]{Sak}). The joint strong$^*$-continuity of the product on bounded sets of $A^{**}$ (see \cite[Proposition 1.8.12]{Sak}) and the strong$^*$-continuity of $T^{**}$ (cf. \cite[Proposition 1.8.10]{Sak}) give $x T^{**} (x) x =0$, for every $x\in A^{**}$. This proves that $T^{**} (1) =0$, and \eqref{eq generalized derivation} implies that $T$ is a derivation.
\end{proof}

\section{Strong-local and weak$^*$-local $^*$-automorphisms}\label{sec: strong-local automorphims on von Neumann algebras}

In this section we shall study strong-local $^*$-automorphisms on von Neumann algebras. Suppose that $T: M\to M$ is a weak$^*$-local $^*$-automorphism on a von Neumann algebra. It is well known that every $^*$-automorphism on $M$ is contractive, so given $\phi\in S_n (M)$ and $a\in M$, there exists a  $^*$-automorphism $\pi_{a,\phi} : M\to M$, such that $\phi T(a) = \phi \pi_{a,\phi} (a)$. Therefore, $|\phi T(a)|\leq \|a\|$, for every $a\in M$ and for every $\phi\in S_n (M)$. This proves that $T$ is bounded. Furthermore, the same argument shows that $T(a)^* = T(a)$ (respectively, $T(a)\geq 0$), whenever $a=a^*$ (respectively, $a\geq 0$), that is, $T^{\sharp} = T$ is a symmetric bounded linear operator on $M$. It is also easy to see that $T(1) = 1$. The main result in this section can be stated now.

\begin{theorem}\label{t strong$^*$-local *-automorphisms} Every strong-local $^*$-automorphism on a von Neumann algebra is a Jordan $^*$-homomorphism.
\end{theorem}

\begin{proof} Let $T: M\to M$ is a strong-local $^*$-automorphism on a von Neumann algebra. We have already commented that $T$ is bounded. Let $u$ be a unitary element in $M$. For each $\phi\in S_n (M)$, there exists a $^*$-automorphism $\pi_{u,\phi} : M\to M$ depending on $u$ and $\phi$ such that $$\||T(u) - \pi_{u,\phi} (u)|\|_{\phi }=0, \hbox{ and hence } \||T(u) |\|_{\phi } = \||\pi_{u,\phi} (u)|\|_{\phi }.$$
Since $\pi_{u,\phi}$ is a $^*$-automorphism and $u$ is a unitary, we get $$\phi\left( T(u)^* T(u) \right) = \||T(u) |\|_{\phi }^2
= \||\pi_{u,\phi} (u)|\|_{\phi }^2 = \phi \left(\pi_{u,\phi} (u)^* \pi_{u,\phi} (u)\right) = \phi (1) =1.$$ It is well known that the normal states on $M$ separate the points in $M$, so, $T(u)^* T(u) = 1$. Replacing $u$ with $u^*$ and having in mind that $T$ is symmetric, we get $T(u) T(u)^* = T(u^*)^* T(u^*) = 1$.\smallskip

We have therefore, proved that $T(u)$ is a unitary in $M$ whenever $u$ is a unitary. So, given a projection $p\in M$, the element $T(1-2 p ) = 1- 2 T(p)$ is a unitary in $M$, which proves that $T(p)$ is a projection of $M$.  Given two orthogonal projections, $p,q\in M$, we know that $T$ maps $p+q$ to a projection, thus, $T(p)^2 + T(q)^2 + T(p) T(q) + T(q) T(p)= T(p+q)^2 = T(p) + T(q)$, and hence $T(p) T(q) + T(q) T(p)=0$. Since $T(p) T(q) =- T(q) T(p)$, we also have $$T(q) T(p) T(q) =- T(q)  T(q) T(p) = -T(q) T(p) $$ and $$T(q) T(p) T(q) = -T(q) T(p) T(q),$$ which gives $T(q) T(p) T(q) =0= T(p) T(q) = T(q) T(p)$. We have shown that $T$ maps orthogonal projections to orthogonal projections. If we approximate every element in $M_{sa}$ by a finite linear combination of mutually orthogonal projections in $M$, it follows from the linearity and continuity of $T$ that $T(a^2 ) = T(a)^2$, for every $a\in M_{sa}$. A standard polarization argument implies that $T$ is a Jordan $^*$-homomorphism.
\end{proof}

Example 3.14 in \cite{Pe2014} shows the existence of a linear bijection $T: M_2 (\mathbb{C}) \to M_2 (\mathbb{C})$, which is a local $^*$-automorphism, and a Jordan $^*$-automorphism, but it is not multiplicative. So, the conclusion of the above theorem is optimal.\smallskip

In \cite{Mol2014}, L. Moln\'{a}r introduced and studied bilocal $^*$-automorphisms on $B(H)$ with a definition inspired by that given by Xiong and Zhu in \cite{ZhuXiong97} for bilocal derivations. A linear mapping $T: B(H) \to B(H)$ is said to be a \emph{bilocal $^*$-automorphism} if for every $a$ in $B(H)$ and every $\xi$ in $H$, there exists a $^*$-automorphism $\pi_{a,\xi}: B(H)\to B(H)$, depending on $a$ and $\xi$, such that $T(a) (\xi) = \pi_{a,\xi} (a) (\xi)$. Inspired by this notion and by our Definition \ref{def extreme weak$^*$-local derivation}, we introduced the following concept:

\begin{definition}\label{def extreme weak$^*$-local $^*$-automorphism} Let $M$ be von Neumann algebra. A linear mapping $T: M \to M$ is said to be an extreme-weak$^*$-local $^*$-automorphism {\rm(}respectively, an extreme-strong-local $^*$-automorphism{\rm)} if for every $a\in M$, and every pure normal state $\phi\in \partial_{e} (S_n(M))$, there exists a $^*$-automorphism $\pi_{a,\phi}: M\to M,$ depending on $a$ and $\phi$, such that $\Big|\phi \Big(T(a) - \pi_{a,\phi} (a)\Big)\Big|=0,$ {\rm(}respectively, $\||T(a) - \pi_{a,\phi} (a)|\|_{\phi }=0${\rm)}.
\end{definition}

Bilocal $^*$-automorphisms and extreme-strong-local $^*$-automorphisms on  $B(H)$ define the same applications.\smallskip

Similar argument to those given in Remark \ref{remark counterexample in the non atomic setting} can be used to find a C$^*$-algebra $A$, and an unbounded linear mapping $T: A^{**}\to A^{**}$ which is an extreme-strong-local $^*$-automorphism but not multiplicative.\smallskip

For an atomic von Neumann algebra $M$, the pure normal states of $M$, $\partial_e (S_n(M))$, separate the points of $M$. Therefore, the proof of Theorem \ref{t strong$^*$-local *-automorphisms} can be easily adapted to establish:

\begin{theorem}\label{t extreme-strong-local *-automorphisms} Every extreme-strong-local $^*$-automorphism on an atomic von Neumann algebra is a Jordan $^*$-homomorphism.$\hfill\Box$
\end{theorem}

Theorems \ref{t strong$^*$-local *-automorphisms} and \ref{t extreme-strong-local *-automorphisms} can be complemented with the following result due to R. Kadison: If $T: M \to N$ is a Jordan $^*$-isomorphism between von Neumann algebras, then there exist weak$^*$ closed ideals $M_1$ and $M_2$
in $M$ and $N_1$ and $N_2$ in $N$ satisfying $M = M_1 \oplus^{\infty} M_2$, $N= N_1 \oplus^{\infty} N_2$, $T|_{M_1} :M_1 \to N_1$ is an $^*$-isomorphism, and $T|_{M_2} : M_2 \to N_2$ is an $^*$-anti-isomorphism (see \cite[Theorem 10]{Kad51}). More general versions of Kadison's theorem can be found in \cite[\S 3]{Stor1965} and \cite{Bre89}.\smallskip

In a very recent contribution (see \cite[Theorem 1]{Mol2014}), L. Moln\'{a}r establishes that for a linear transformation $T: B(H)\to B(H)$, where $H$ is an infinite dimensional and separable complex Hilbert space, the following two assertions are equivalent:\begin{enumerate}[$(a)$]\item $T$ is a bilocal $^*$-automorphism, that is, for every $a\in B(H)$ and $\xi \in H$, there exists an algebra $^*$-automorphism $\pi_{a,\xi}$ of $B(H)$ such that $T(a) (\xi) = \pi_{a,\xi} (a) (\xi)$;
\item $T$ is a unital algebra $^*$-endomorphism of $B(H)$.
\end{enumerate}

We can state now a direct consequence of Theorem \ref{t extreme-strong-local *-automorphisms} and the previously commented result of Kadison. We shall show that \cite[Theorem 1]{Mol2014} admits a Jordan version for general complex Hilbert spaces.

\begin{corollary}\label{c generalization of Molnar bilocal *-automorphisms} Let $T: B(H)\to B(H)$ be a linear mapping, where $H$ is a complex Hilbert space. Suppose $T$ is a bilocal $^*$-automorphisms or a extreme-strong-local $^*$-automorphism, that is, for every $a\in B(H)$ and $\xi \in H$, there exists an algebra $^*$-automorphism $\pi_{a,\xi}$ of $B(H)$ such that $T(a) (\xi) = \pi_{a,\xi} (a) (\xi)$. Then $T$ is a unital Jordan $^*$-endomorphism of $B(H)$.$\hfill\Box$
\end{corollary}

We finish this note with a reflection on weak-local $^*$-automorphisms on $C(\Omega)$-spaces. By the Banach-Stone theorem, every $^*$-automorphism $\pi$ of $C(\Omega)$ is a composition operator of the form $\pi (f) (t) = f(\sigma(t))$ ($t\in \Omega$), where $\sigma: \Omega\to \Omega$ is a homeomorphism. Therefore, if $T: C(\Omega)\to C(\Omega)$ is a weak-local $^*$-automorphism, then for each $t\in \Omega$ and each function $f\in C(\Omega),$ there exists a homeomorphism $\sigma_{f,t}: \Omega\to \Omega$, depending on $f$ and $t$ such that $$T(f) (t) = f(\sigma_{f,t}(t))\in f(\Omega).$$ The Gleason-Kahane-\.{Z}elazko theorem (cf. \cite{Gle,KaZe}) assures that $T$ is multiplicative and a $^*$-homomorphism.

\end{document}